\documentclass[11pt]{amsart}

\usepackage[utf8]{inputenc}
\usepackage[english]{babel}

\usepackage{amssymb}
\usepackage{bbm}
\usepackage{amsmath}

\usepackage{mathbbol}

\usepackage{graphics}

\usepackage{amsthm}

\newtheorem{theorem}{Theorem}[section]
\newtheorem{lemma}[theorem]{Lemma}
\newtheorem{corollary}[theorem]{Corollary}

\theoremstyle{definition}
\newtheorem{example}[theorem]{Example}

\newtheorem{definition}[theorem]{Definition}

\DeclareMathOperator{\inter}{int}

\DeclareMathOperator{\conv}{conv}

\DeclareMathOperator{\vol}{vol}
\DeclareMathOperator{\cone}{cone}
\DeclareMathOperator{\cl}{cl}
\DeclareMathOperator{\aff}{aff}

\DeclareMathOperator{\V}{V}

\DeclareMathOperator{\R}{\mathbbm{R}}

\DeclareMathOperator{\Z}{\mathbbm{Z}}
\DeclareMathOperator{\Q}{\mathbbm{Q}}
\DeclareMathOperator{\0}{\mathbb{0}}

\DeclareMathOperator{\Grat}{Q}
\DeclareMathOperator{\den}{den}

\DeclareMathOperator{\denrat}{\widehat{den}}

\DeclareMathOperator{\coloneq}{\mathrel{\mathop:}=}
\DeclareMathOperator{\eqcolon}{=\mathrel{\mathop:}}

\newcommand{\gauss}[1]{\left\lfloor #1\right\rfloor}

\newcommand{\fracpart}[1]{\left\{#1\right\}}
\newcommand{\e}[1]{e(#1)}
\newcommand{\E}[2]{E(#1#2)}

\numberwithin{equation}{section}

\title{Lattice points in vector-dilated polytopes}
\author{Martin Henk}
\author{Eva Linke}
\address{Martin Henk and Eva Linke, Institut f\"ur Algebra und Geometrie, Universit\"at Mag\-deburg, Universit\"atsplatz 2, D-39106-Magdeburg,
Germany} \email{martin.henk@ovgu.de, eva.linke@ovgu.de}
\thanks{Second author is supported by the Deutsche Forschungsgemeinschaft within the
project He 2272/5-1.}

\begin{document}

\begin{abstract}
For $A\in\mathbb{Z}^{m\times n}$ we investigate the behaviour of the number of lattice points in $P_A(b)=\{x\in\mathbb{R}^n:Ax\leq b\}$, 
depending on the varying vector $b$. It is known that this number, restricted to a cone of constant combinatorial type of 
$P_A(b)$, is a quasi-polynomial function if b is an integral vector.
We extend this result to rational vectors $b$ and show that the coefficients themselves are piecewise-defined polynomials.
To this end, we use a theorem of McMullen on lattice points in Minkowski-sums of rational dilates of rational polytopes 
and take a closer look at the coefficients appearing there.
\end{abstract}

\maketitle
\section{Introduction}
Let $\R^n$ be the $n$-dimensional Euclidean space and $\Z^n$ the integral lattice.
By $\e{i}$, we denote the $i$th coordinate unit-vector, that is $\e{i}_{k}=0$ for all
$k\neq i$ and $\e{i}_i=1$. By $\E{i}{j}\in\R^n$, we denote the matrix with $\E{i}{j}_{kl}=0$
for $(k,l)\neq(i,j)$ and $\E{i}{j}_{ij}=1$. The origin of an appropriate dimension
is denoted by $\0=(0,\ldots,0)$.

A polytope is the convex hull $\conv\{v_1,\ldots,v_k\}$ of points $v_1,\ldots,v_k\in\R^n$.
It is called rational, if $v_1,\ldots,v_k$ can be chosen to be in $\Q^n$.
Equivalently, $P$ is a rational polytope, if and only if, there are $A\in\Q^{m\times n}$ and
$b\in\Q^m$ with $P=\{x\in\R^n:Ax\leq b\}$. It is clear, that $A$ can also be chosen to
be an integral matrix.
For a rational polytope $P$ we call the smallest positive integral (resp.\ rational) number $d$ such that $dP$ is an integral
polytope the \emph{(rational) denominator of $P$} and denote it by $\den(P)$ (resp. $\denrat(P)$).
Furthermore, $\vol(P)=\vol_n(P)$ denotes the volume of $P$, that is its $n$-dimensional
Lebesgue measure and $\vol_{\dim(P)}(P)$ is the Lebesgue-measure of $P$ with respect to its affine hull.

For a given integral 
$(m\times n)$-matrix $A$, let $P_A(b)\coloneq\{x\in\R^n:Ax\leq b\}$, $b\in\Q^m$. 
In this work we want to count the number 
of lattice points in $P_A(b)$ as a function in $b$ for a fixed matrix 
$A$. 

To this end, we consider matrices $A\in\Z^{m\times n}$ such that
$P_A(b)$ is bounded for all $b\in\Q^m$, that is, $\cone(A^\top)=\R^n$.
Here, $\cone(A^\top)$ is the cone generated by the rows of $A$, that is, the set of
all nonnegative linear combinations of rows of $A$.

We denote the number of lattice points in $P_A(b)$ by
\begin{equation*}
\Phi(A,b):=\#(P_A(b)\cap\Z^n),\quad b\in\Q^m.
\end{equation*}

Since we cannot expect uniform behavior of $\Phi(A,b)$ when the 
polytope $P_A(b)$ changes combinatorically, we consider subsets of $\Q^m$
on which the combinatorial structure of $P_A(\cdot)$ is constant. For this, we have to
consider the possible normal fans of $P_A(b)$. 

For a fixed vertex $v$ of a polytope $P_A(b)$, the \emph{normal cone} $\tau_v$ of 
$v$ is the set of all directions $u\in\R^n$, such that the function 
$x\mapsto u^\top x$, $x\in P_A(b)$, is maximized by $v$. By the definition of 
vertices as $0$-dimensional faces,
the normal cone $\tau_v$ of a vertex $v$ is full-dimensional.

We call the set of the normal cones of all vertices of $P_A(b)$ the \emph{normal 
fan}, denoted by $N_A(b)$.
Observe that this differs from the usual notion of the normal fan, which is a
polyhedral subdvision of $\R^n$ and hence contains also
lower dimensional normal cones. In our case it is enough to consider only 
the maximal cells. Still, the union of all normal cones in
$N_A(b)$ is $\R^n$, and the interiors of two normal cones in $N_A(b)$ do
not intersect. We refer to Ziegler \cite[Chapter 7]{ZIEGLER} for an introduction
to polyhedral fans.

For a given matrix $A$ and varying $b$, there are only finitely many possible normal fans and
the normal fan fixes the combinatorial structure of $P_A(b)$. Hence, in the 
following we will always fix the normal fan $N$. 
For a fixed normal fan $N$, let $C_N\subset\Q^m$ be the set of
all vectors $b$ such that $N_A(b)=N$.
The set $\{C_N: N=N_A(b) \text{ for some } b\in\Q^m\}$ is denoted by 
$\mathcal{C}_A$. Then for every $C\in\mathcal{C}_A$, its closure $\cl(C)$ is a polyhedral cone (cf.\ Lemma\nobreakspace \ref {lem:right_side_cones}). 
To state the main result in a comprehensive way, we fix some abbreviatory 
notation: For $x\in\R^n$, $y\in\Z_{\geq 0}^n$ we write 
$x^y:=\prod_{j=1}^nx_j^{y_j}$. For a $k$-tupel $I=(I_1,\ldots,I_k)\in\{0,\ldots,n\}^k$ 
we denote by $|I|_1=\sum_{j=1}^kI_j$  the usual $1$-norm.

\begin{theorem}\label{thm:rqp_right_sides}
Let $C=\cone\{h_1,\ldots,h_k\}\in\mathcal{C}_A$. 
Then $\Phi(A,b)$ is a quasi-polynomial function in $b\in\cl(C)\cap\Q^m$, that is,
\begin{equation*}
\Phi(A,b)=\sum_{\substack{J\in\{0,\ldots,n\}^k\\ |J|_1\leq n}}\Phi_J(A,b)b^J,
\end{equation*}
where $\Phi_J(A,b)=\Phi_J(A,b+\denrat(P_A(h_i))h_i)$ for all $J$ and $i$. Furthermore,
$\Phi_J(A,b)$ is a piecewise-defined polynomial of total degree $n-|J|_1$ in $b$
with
\begin{equation*}
\frac{\partial}{\partial b_l}\Phi_J(A,b)=
-(J_l+1)\Phi_{J+\e{l}}(A,b),
\end{equation*}
\end{theorem}
The proof of this theorem as well as an example can be found in Section\nobreakspace \ref {sec:vector_dilations}.

Dahmen and Micchelli, 1988, \cite[Theorem 3.1]{Dahmen&Micchelli} gave a structural
result for
\begin{equation*}
\Phi^=(A,b)\coloneq\#(\{x\in\R_{\geq 0}^n: Ax=b\}\cap\Z^n)
\end{equation*}
for a fixed matrix 
$A$ and suitable $b$ also inside cones $C\in\mathcal{C_A}$ if 
$b$ is an integral polytope. As a corollary \cite[Corollary 3.1]{Dahmen&Micchelli}, they get 
that $\Phi^=(A,\cdot)$ is a polynomial
in the integral variable $b\in C$, if $P^=_A(b)=\{x\in\R_{\geq 0}^n: Ax=b\}$ is integral.
Sturmfels, 1995, \cite{Sturmfels95} gave a formula for the difference between these polynomials 
and $\Phi^=(A,b)$, if $P^=_A(b)$ is not integral. The works make use of the theory of 
polyhedral splines and representation techniques of groups.
Mount, 1998, \cite{Mount} described methods for actually calculating the 
polynomials and cones, if $A$ is unimodular and $b$ integral. To this end,
Mount gave an alternative argument for \cite[Corollary 3.1]{Dahmen&Micchelli}
which we also follow in Section\nobreakspace \ref {sec:vector_dilations}.
Beck \cite{Beck1999,Beck2003} gave a more elementary proof of the quasi-polynomiality 
of $\Phi(A,b)$, if $b$ is integral. He also proved an Ehrhart reciprocity law for
vector dilated polytopes, that is, $\Phi(A,-b)=\#(\inter(P_A(b))\cap\Z^n)$, for $b\in\Z^n$.
Here, $\inter(P)$ denotes the interior of a polytope $P$.
Since $\Phi(A,b)=\Phi(tA,tb)$ for all $t\in\Q_{\geq 0}$, $A\in\Z^{m\times n}$ and
$b\in\Q^m$, this statement immediately carries over to rational vectors $b$ and we have
\begin{corollary}
\[\Phi(A,-b)=\#(\inter(P_A(b))\cap\Z^n)\]
for all $b\in\Q^m$.
\end{corollary}

To prove Theorem\nobreakspace \ref {thm:rqp_right_sides}, we first consider for rational polytopes 
$P_1,\ldots,P_k\subset\R^n$, $k\in\Z_{\geq 1}$ the function 
$\Grat(P_1,\ldots,P_k,\cdot):\Q_{\geq 0}^k\to\Z_{\geq 0}$ given by
\begin{equation*}
\Grat(P_1,\ldots,P_k,r):=\#\left(\left(\sum_{i=1}^kr_iP_i\right)\cap\Z^n\right),\quad
\text{for }r=(r_1,\ldots,r_k)\in\Q_{\geq 0}^k.
\end{equation*}
 

For $k=1$ such considerations go back to Ehrhart.
A function $p:\Q_{\geq 0}\to\R$ is called a \emph{quasi-polynomial with period $d$}
of degree at most $n$ if there exist periodic functions $p_i:\Q_{\geq 0}\to\R$,
$i=1,\ldots,n$, with period $d$ such that $p(r)=\sum_{i=0}^np_i(r)r^i$.

Ehrhart's Theorem \cite{EHRHART1962} states, that $\#(kP\cap\Z^n)$ is a quasi-polynomial
with period $\den(P)$ of degree $\dim(P)$ for $k\in\Z_{\geq 1}$.
The leading coefficient if this quasi-polynomial is $\vol_{\dim(P)}(P)$ for all $k\in\Z_{\geq 1}$ such that
$\aff(kP)$ contains integral points. Thus, if $P$ is full-dimensional, 
the leading coefficient is constant and equals $\vol(P)$ for all $k\in\Z_{\geq 1}$.
For more information on Ehrhart theory we refer to Beck and Robins \cite{Beck&Robins2006}.

This was recently generalized to rational dilation factors by Linke \cite{Linke}:

\begin{theorem}[{\cite{Linke,McMullen1978}}]\label{thm:my_rational_ehrhart}
Let $P$ be a rational polytope in $\R^n$. Then 
\begin{equation*}
\#(rP\cap\Z^n)=\sum_{i=0}^{\dim(P)}\Grat_i(P,r)r^i,\quad\text{for }r\in\Q_{\geq 0}
\end{equation*}
is a quasi-polynomial with period $\denrat(P)$ of degree $\dim(P)$.
Furthermore, 
$\Grat_i(P,\cdot)$ is a piecewise-defined polynomial of degree $n-i$, and 
\begin{equation*}
\Grat_i'(P,r)=-(i+1)\Grat_{i+1}(P,r),\quad i=0,\ldots,n-1,
\end{equation*} 
for all $r\geq 0$ such that $\Grat(P,\cdot)$ is (one-sided) continuous at
$r+k\denrat(P)$ for all $k\in\Z_{\geq 0}$
\end{theorem}
We have $\Grat_0(P,0)=1$ and $\Grat_{\dim(P)}(P,r)=\vol_{\dim(P)}(P)$ for all $r\in\Q_{>0}$ such that
$\aff(rP)$ contains integral points.
Baldoni et al, 2010, \cite{Koeppe2010} study intermediate 
sums, interpolating between integrals and discrete sums over certain integral points in 
polytopes which results in a rational version of Ehrhart' Theorem for such intermediate
valuations. 

In the next section, we generalize these results to $k>1$.

\begin{definition}[Rational Quasi-polynomial in several unknowns]
A function $p:\Q_{\geq 0}^k\to\Q$ is called a \emph{rational quasi-polynomial of total 
degree $n$ with period $d=(d_1,\ldots,d_k)\in\Q^k$} if there exist periodic 
functions $p_I:\Q^k_{\geq 0}\to\Q$ for all $I\in\{0,\ldots,n\}^k$ with period $d_i$
in the $i$th component, $1\leq i\leq k$, such that
\begin{equation*}
p(r)=\sum_{\substack{I\in\{0,\ldots,n\}^k\\ |I|_1\leq n}}p_I(r)r^I.
\end{equation*}
We call $p_I(\cdot)$ the $I$th coefficient of $p$.
\end{definition}

The following theorem generalizes the univariate case to Minkowski sums of polytopes
and follows actually from McMullens proof of Theorem $7$ in \cite{McMullen1978} although not stated 
explicitely there.

\begin{theorem}[{\cite{McMullen1978}}]\label{thm:my_rational_ehrhart_mink_sum}
Let $P_1,\ldots,P_k$ be rational polytopes in $\R^n$. Then 
$\Grat(P_1,\ldots,P_k,\cdot)$
is a rational quasi-polynomial of  total degree $\dim(P_1+\ldots+P_k)$ with period 
$d=(\denrat(P_1),\ldots,\denrat(P_k))$.
$\Grat(P_1,\ldots,P_k,\cdot)$ is called the \emph{rational Ehrhart quasi-polynomial 
of $P_1,\ldots,P_k$}. 
The $I$th coefficient of $\Grat(P_1,\ldots,P_k,\cdot)$ is denoted by 
$\Grat_I(P_1,\ldots,P_k,\cdot)$. 
\end{theorem}
Further,
\begin{equation*}
\Grat_I(P_1,\ldots,P_k,r)=\frac{m!}{I_1!\cdots I_k!}\V_I(P_1,\ldots,P_k)
\end{equation*}
for all $I\in\{0,\ldots,n\}^{k}$, with 
$|I|_1=\dim(P_1+\ldots+P_k)=:m$ and for all $r\in\Q_{>0}^k$ with
$\aff\left(\sum_{i=1}^kr_iP_i\right)$ contains integral points.
Here, $\V_I(P_1,\ldots,P_k)$ is the $I$th mixed volume of $P_1,\ldots,P_k$. For  
more information about mixed volumes we refer to Schneider \cite{SCHNEIDER93}. Briefly summarized,
the $m$-dimensional volume of $\sum_{i=1}^kr_iP_i$ is a homogenous polynomial of degree $m$.
It coefficients depend only on the involved polytopes and are up to a constant of $\frac{m!}{I_1!\cdots I_k!}$
called mixed volumes and denoted by $\V_I(P_1,\ldots,P_k)$. 
As in the univariate case, the volume of $r_1P_1+\ldots+r_kP_k$ is the leading
term of $\Grat(P_1,\ldots,P_k,r)$, where the leading term of a multivariate quasi-polynomial
of degree $m$, 
$p(x)=\sum_{|I|_1\leq m}p_I(x)x^I$ is $\sum_{|I|_1=m}p_I(x)x^I$.
Hence the mixed volumes appear as coefficients of $\Grat(P_1,\ldots,P_k,\cdot)$.

Again, we can show more about the coefficients. In what follows, we denote by $r\odot s:=(r_1s_1,\ldots,r_ks_k)$ the 
componentwise multiplication of $r,s\in\Q^k$. Then in analogy to Theorem\nobreakspace \ref {thm:my_rational_ehrhart} we have

\begin{theorem}\label{thm:rqp_polycoeffs_diff_mink_sums}
Let $P_1,\ldots,P_k$ be rational polytopes in $\R^n$ and let 
$\dim(P_1+\ldots+P_k)=n$. 
Then for $I\in\{0,\ldots,n\}^k$, $\Grat_I(P_1,\ldots,P_k,\cdot)$ 
is a piecewise-defined polynomial function of total 
degree $n-|I|_1$ and of degree $n-I_j$ in $r_j$, and
\begin{equation*}
\frac{\partial}{\partial r_j}\Grat_I(P_1,\ldots,P_k,r)=
-(I_j+1)\Grat_{I+e(j)}(P_1,\ldots,P_k,r),
\end{equation*}
for all $|I|_1<n$ and for all $r\in \Q_{\geq 0}$ such that 
$\Grat(P_1,\ldots,P_k,\cdot)$ is continuous at
$r+U\odot(\denrat(P_1),\ldots,\denrat(P_k))^\top$ for all $U\in\Z_{\geq 0}^k$.
\end{theorem}
The proof of this theorem is given in the next section.
\section{Proof of Theorem\nobreakspace \ref {thm:rqp_polycoeffs_diff_mink_sums}}\label{sec:rqp_mink_sums}
We first fix some more notation: For a facet $F$ of a rational polytope $P$,
$\alpha(F)$ denotes the smallest positive rational number such that 
$\alpha(F)\aff(F)$ contains integral points. For a rational number $r\in\Q$ we denote
by $\fracpart{r}$ the fractional part of $r$, that is, $\fracpart{r}=r-\gauss{r}$
where $\gauss{\cdot}$ is the floor-function. 

For rational polytopes $P_1,\ldots,P_k\subset\R^n$ 
all possible Minkowski-sums $r_1P_1+\ldots+r_kP_k,$ $r\in\Q_{> 0}^k$ 
have the same normal fan.
This allows us to number the facets of all polytopes $r_1P_1+\ldots+r_kP_k$ consistently and
we can denote the facets of the polytope $r_1P_1+\ldots+r_kP_k$ by 
$F_1(r),\ldots,F_l(r)$.

With this notation, McMullen's work implies something stronger than Theorem\nobreakspace \ref {thm:my_rational_ehrhart_mink_sum}:
\begin{theorem}[McMullen, 1978 {\cite{McMullen1978}}]
$\Grat_I(P_1,\ldots,P_k,r)$ 
depends only on $P_1,\ldots,P_k$ and on the values
$\fracpart{\frac{1}{\alpha_{F_i(r)}} }$ for $i=1,\ldots,l$, but not on $\{r_j/\denrat(P_j)\}$, $j=1,\ldots,k$. 
%
\end{theorem}

For $r,s\in\Q_{\geq 0}^k$ we say, $r$ is equivalent to $s$ if $r_1P_1+\ldots+r_kP_k=s_1P_1+\ldots+s_kP_k$.
This relation is an equivalence relation and we denote the equivalence class of $r\in\Q_{\geq 0}^k$
by $[r]$. From McMullens Theorem immediately follows

\begin{corollary}\label{rem:alpha_F}
The coefficients of the rational Ehrhart quasi-polynomial of $P_1,\ldots,P_k$ 
are constant on $[r]$.
\end{corollary}

\begin{example}
As an example, we consider $C_2=\conv\{\binom{1}{1},\binom{-1}{1},\binom{-1}{-1},
\binom{1}{-1}\}$ and the triangle $T=\{\binom{0}{1},\binom{1}{-1},\binom{-1}{-1}\}$
(see Figure\nobreakspace \ref {fig:triangle_cube}).
\begin{figure}[hbt]\centering
\includegraphics{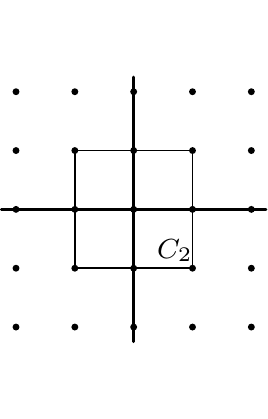}
\hspace{1mm}
\includegraphics{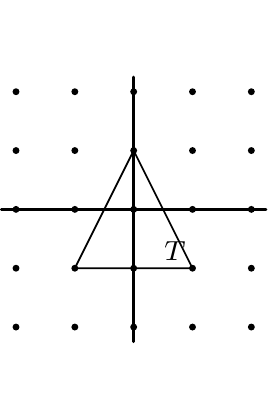}
\hspace{1mm}
\includegraphics{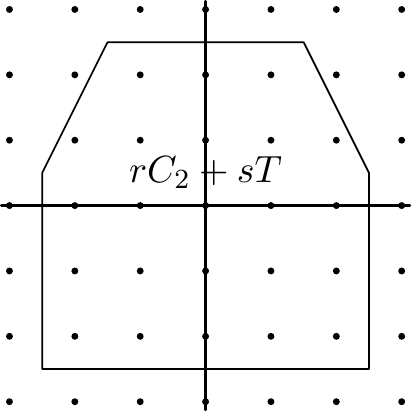}
\caption{$C_2$ and $T$ and $rC_2+sT$.}\label{fig:triangle_cube} 
\end{figure}
For $r,s\in\Q_{\geq 0}$, the sum $rC_2+sT$ has the following structure:
\begin{equation*}
\begin{split}
rC_2+sT&=\conv\left\{\binom{r}{r+s},\binom{-r}{r+s},
\binom{-(r+s)}{r-s},\binom{-(r+s)}{-(r+s)},\binom{(r+s)}{-(r+s)},\binom{r+s}{r-s}
\right\}\\
&=\{x\in\R^2:\quad\begin{array}[t]{rcl}-(r+s)\leq& x_1&\leq r+s,\\
-(r+s)\leq& x_2&\leq r+s,\\&\pm2x_1+x_2&\leq3r+s\}.\end{array} 
\end{split}
\end{equation*}
Thus for the edges (facets) $F$ of $rC_2+sT$ it holds that $\frac1{\alpha_F}\in\{r+s,3r+s\}$, 
and hence the coefficients of the rational Ehrhart quasi-polynomial depend only on
$\fracpart{r+s}$ and $\fracpart{3r+s}$ (see Corollary\nobreakspace \ref {rem:alpha_F}).

The coefficients as functions in $\fracpart{r+s}$ and 
$\fracpart{3r+s}$ are:
\begin{equation*}\begin{split}
\Grat_{(2,0)}(C_2,T,r,s)&=4,\\\Grat_{(1,1)}(C_2,T,r,s)&=8,\\
\Grat_{(0,2)}(C_2,T,r,s)&=2,\\
\Grat_{(1,0)}(C_2,T,r,s)&=-8\fracpart{r+s}+4,\\
\Grat_{(0,1)}(C_2,T,r,s)&=-2\fracpart{3r+s}-2\fracpart{r+s}+2,\\
\Grat_{(0,0)}(C_2,T,r,s)&=-\frac12\left(\fracpart{3r+s}^2+\fracpart{r+s}^2\right)
+3\fracpart{3r+s}\fracpart{r+s}
-\fracpart{r+s}\\&\qquad-\fracpart{3r+s}+1-\begin{cases}\frac12,&\text{if }
\fracpart{3r+s}-\fracpart{r+s}-2r\text{ odd},\\
0,&\text{otherwise}.\end{cases}
\end{split}\end{equation*}
\end{example}

To prove Theorem\nobreakspace \ref {thm:rqp_polycoeffs_diff_mink_sums} the following lemma is
the main step.
When considering $k$-tuples in $\{0,\ldots,n\}^k$, we allow them
to be added componentwise.
\begin{lemma}\label{lem:rqp_coeffs_poly_mink_sums}
Let $p:\Q^k\to\Q$ be a rational quasi-polynomial of total degree $n\in\Z_{\geq 1}$ 
with period $d\in\Q^k_{>0}$ and constant leading coefficients, that is,
\begin{equation} \label{eq:qp_form}
p(r)=\sum_{\substack{I\in\{0,\ldots,n\}^k\\ |I|_1\leq n}}p_I(r)r^I,
\quad r=(r_1,\ldots,r_k)\in\Q^k,
\end{equation}
where $p_I(r)\eqcolon p_I\in\Q$ for $|I|_1=n$, and $p_I:\Q^k\to\Q$ are periodic 
functions 
with period $d$ for $I\in \{0,\ldots,n\}^k$, $|I|_1<n$. Furthermore, suppose there 
exist a set $S\subset \Q^k$ and $c_U\in\Q$ for all $U\in\Z^k_{\geq 0}$ such that
\begin{equation*}
p(r+U\odot d)=c_U,\quad\text{for all }r\in S,\ U\in\Z^k_{\geq 0}.
\end{equation*}
Then $p_I:S\to\Q$ is a polynomial of total degree $n-|I|_1$ and of degree at most 
$n-I_j$ in $r_j$ with
\begin{equation*}
\frac{\partial}{\partial r_j}p_I(r)=-(I_j+1)p_{I+\e{j}}(r),
\end{equation*}
for all $I\in\{0,\ldots,n\}^k, |I|_1<n$ and for all $r\in S$.
\end{lemma} 

As in the univariate case in \cite{Linke} the proof is by induction on the total degree. For 
simplification, we subdivide a part of the induction step which
gives the following lemma:

\begin{lemma}\label{lem:inductive_step_rqp_mink_sum}
For all $I\in\{0,\ldots,n\}^k$ with $|I|_1\leq n-1$ and $l\in\{1,
\ldots,k\}$ and for a fixed subset $S\subset \Q^k$ let $q_I^l:S\to\Q$ and 
$p_I:S\to\Q$ with
\begin{equation*}
q^l_I(r)=\sum_{j=I_l+1}^{n-|I|_1+I_l}p_{(I_1,\ldots,I_{l-1},j,I_{l+1},\ldots,
I_k)}(r)\,c^l_I(j),
\end{equation*}
for constants $c_I^l(j)$.
Furthermore suppose $q^l_I$ is a polynomial of total degree $n-1-|I|_1$ in $r$ and 
of  degree $n-1-I_h$ in $r_h$ for all $h\in\{1,\ldots,k\}$.
Then for $J\in\{0,\ldots,n\}^k$ with $1\leq |J|_1\leq n$, $p_{J}$ 
is a polynomial of total degree $n-|J|_1$ in 
$r$ and of degree at most $n-J_h$ in $r_h$ for all $h\in\{1,\ldots,k\}$.
\end{lemma}

\begin{proof}
We show this by induction on $|J|_1$.\\ For $|J|_1=n$ the statement 
is clear, since for $J_k\neq 0$, say, we have that 
\begin{equation*}
q^k_{(J_1,\ldots,J_{k-1},J_k-1)}(r)=p_{J}(r)\,c^k_{(J_1,\ldots,J_{k-1},J_k-1)}(J_k).
\end{equation*}
For $|J|_1<n$ and again, without loss of generality, $J_k\neq 0$, consider 
$q^k_{J-\e{k}}$, which is
\begin{equation*}\begin{split}
q^k_{J-\e{k}}(r)&=\sum_{i=J_k}^{n-|J|_1+J_k}p_{J+(i-J_k)\e{k}}(r)
\,c^k_{J-\e{k}}(i)\\
&=p_{J}(r)\,c^k_{J-\e{k}}(J_k)+\sum_{i=J_k+1}^{n-|J|_1+J_k}p_{J+(i-J_k)\e{k}}(r)
\,c^k_{J-\e{k}}(i).
\end{split}\end{equation*}
Thus 
\begin{equation*}\begin{split}
p_{J}(r)\,c^k_{J-\e{k}}(J_k)=q^k_{J-\e{k}}(r)-\sum_{i=J_k+1}^{n-|J|_1+J_k}
p_{J+(i-J_k)\e{k}}(r)\,c^k_{J-\e{k}}(i).
\end{split}\end{equation*}
By induction hypothesis for $i>J_k$, $p_{J+(i-J_k)\e{k}}(r)$ is a polynomial of total 
degree $n-|J|_1-i+J_k\leq n-|J|_1-1$ in $r$, of degree $n-i\leq n-J_k-1$ in $r_k$, 
and of degree $n-J_h$ in $r_h$ for all $h\in\{1,\ldots,k-1\}$ and $r\in S$. 
Furthermore, $q^k_{J-\e{k}}(r)$ is a polynomial of total degree $n-|J|_1$ in $r$, of 
degree $n-J_k$ in $r_k$, and of degree $n-1-J_h$ in $r_h$ for all 
$h\in\{1,\ldots,k-1\}$ and $r\in S$.

Thus $p_{J}(r)$ is a polynomial of total degree $n-|J|_1$ in $r$, of degree 
$n-J_k$ in $r_k$, and of degree $n-J_h$ in $r_h$ for all $h\in\{1,\ldots,k-1\}$ 
and $r\in S$.
\end{proof}

\begin{proof}[{Proof of Lemma\nobreakspace \ref {lem:rqp_coeffs_poly_mink_sums}}]
We prove the polynomiality result by induction on $n$. 

For $n=1$, we have 
$c_{\0}=p(r)=\sum_{|I|_1=1} p_I\cdot r^I+p_{\0}(r)$ for all 
$r\in S$.
Thus $p_{\0}(r)=c_{\0}-\sum_{|I|_1=1} p_I\cdot r^I$ for 
$r\in S$, which is a polynomial of total degree $n-0=1$ and degree $n-0=1$ in 
$r_h$ for all $h\in\{1,\ldots,k\}$ and $r\in S$.

Now let $n>1$. Consider $q(r)\coloneq p(r+d_k\e{k})-p(r)$.\\
Then $q(r+U\odot d)=p(r+(U+\e{k})\odot d)-p(r+U\odot d)=c_{U+\e{k}}-c_U$ for all
$r\in S$, $U\in Z_{\geq 0}^k$.

To shorten the notation in this 
proof, $I=(I_1,\ldots,I_k)$ and $J=(J_1,\ldots,J_k)$ 
are always vectors in $\{0,\ldots,n\}^k$
and we write $I\leq J$ if $I_l\leq J_l$ for all $l=1,\ldots,k$. 
Furthermore, for $r=(r_1,\ldots,r_k)\in S$ and 
$I\in\{0,\ldots,n\}^k$ we denote by 
$\bar{r}=(r_1,\ldots,r_{k-1})$ and $\bar{I}=(I_1,\ldots,I_{k-1})$, 
respectively, the vector with the last coordinate removed.

Then we get
\begin{equation*}\begin{split}
q(r)&=p(\bar{r},r_k+d_k)-p(r)=\sum_{|I|_1\leq n}p_{I}(r)
\left(\bar{r}^{\bar{I}}(r_k+d_k)^{I_k}-r^I\right)\\
&=\sum_{|I|_1\leq n}p_{I}(r)\left(
\sum_{j=0}^{I_k-1}\binom{I_k}{j}d_k^{I_k-j}\bar{r}^{\bar{I}}r_k^{j}\right)\\
&=\sum_{|\bar{I}|_1\leq n}\sum_{j=0}^{n-|\bar{I}|_1-1}\left(
\sum_{I_k=j+1}^{n-|\bar{I}|_1}p_{(\bar{I},I_k)}(r)\binom{I_k}{j}
d_k^{I_k-j}\right)\bar{r}^{\bar{I}}r_k^j.
\end{split}\end{equation*}
Interchanging the roles of $I_k$ and $j$ yields
\begin{equation*}\begin{split}
q(r)&=
\sum_{|\bar{I}|_1\leq n}\sum_{I_k=0}^{n-|\bar{I}|_1-1}\left(
\sum_{j=I_k+1}^{n-|\bar{I}|_1}
p_{(\bar{I},j)}(r)\binom{j}{I_k}
d_k^{j-I_k}\right)r^I\\
&=\sum_{|I|_1\leq n-1}\left(\sum_{j=I_k+1}^{n-|I|_1+I_k}p_{(\bar{I},j)}(r)
\binom{j}{I_k}d_k^{j-I_k}\right)r^I.
\end{split}\end{equation*}
Thus $q:\Q^k\to\Q$ is a rational quasi-polynomial of total degree $n-1$ with 
period $d$ and coefficients
\begin{equation*}
q_{I}(r)\coloneq\sum_{j=I_k+1}^{n-|I|_1+I_k}p_{(\bar{I},j)}(r)
\binom{j}{I_k}d_k^{j-I_k}.
\end{equation*}
The leading coefficient for $|I|=n-1$ is
\begin{equation*}
q_{I}(r)=\sum_{j=I_k+1}^{n-|I|_1+I_k}p_{(\bar{I},j)}(r)\binom{j}{I_k}d_k^{j-I_k}
=p_{I+\e{k}}(r)\cdot(I_k+1)d_k,
\end{equation*}
which is constant in $r$.
Thus by induction hypothesis we get that
\begin{equation*}
q_{I}(r)=\sum_{j=I_k+1}^{n-|I|_1+I_k}p_{(\bar{I},j)}(r)\binom{j}{I_k}d_k^{j-I_k}
\end{equation*}
is a polynomial of total degree $n-1-|I|_1$ in $r$ and of degree $n-1-I_h$ in 
$r_h$ for all $h\in\{1,\ldots,k\}$ and $r\in S$.
By renaming variables, this is also true if we replace $r_k$ and $I_k$ by 
arbitrary $r_j$ and $I_j$. Thus for all $I=(I_1,\ldots,I_k)$,
\begin{equation*}
q^l_I(r):=\sum_{j=I_l+1}^{n-|I|_1+I_l}
p_{(I_1,\ldots,I_{l-1},j,I_{l+1},\ldots,I_k)}(r)\binom{j}{I_l}d_l^{j-I_l}
\end{equation*}
is a polynomial of total degree $n-1-|I|_1$ in $r$ and of degree $n-1-I_h$ in
$r_h$ for all $h\in\{1,\ldots,k\}$ and $r\in S$.
Thus $p_{I}(r)$ is a polynomial of total degree $n-|I|_1$ in $r$ and of degree 
$n-I_h$ in $r_h$  for all $h\in\{1,\ldots,k\}$ and $r\in S$ by 
Lemma\nobreakspace \ref {lem:inductive_step_rqp_mink_sum}.
For finishing the inductive step it remains to show that $p_{\0}(r)$ 
is a polynomial of total degree $n$ in $r$ and of degree $n$ in $r_h$ for all 
$h\in\{1,\ldots,k\}$ and $r\in S$, which follows since
\begin{equation*}
c_{\0}=p_{\0}(r)+\sum_{0<|I|_1\leq n}p_I(r)r^I.
\end{equation*}
It remains to show that
\begin{equation*}
\frac{\partial}{\partial r_j}p_I(r)=-(I_j+1)p_{I+\e{j}}(r),
\end{equation*}
for all $I=(I_1,\ldots,I_k)\in\{0,\ldots,n\}^k, |I|_1<n, I_j<n$, and for all 
$r\in S$.
To this end, since $p_I(r)$ is a polynomial of total degree $n-|I|_1$ in $r$, we 
can write it as
\begin{equation*}
p_I(r)=\sum_{\substack{J\in\{0,\ldots,n\}^k\\|J|_1\leq n-|I|_1}}p_{I,J}r^J,
\end{equation*}
for some coefficients $p_{I,J}$.
Substituting this in Equation\nobreakspace \textup {(\ref {eq:qp_form})} yields (cf.\ proof of Theorem 1.7 in \cite{Linke}) 
the following form for all $\tilde{r}\in S$
\begin{equation*}
p_J(\tilde{r})=\sum_{\substack{|I|_1\leq n-|J|_1}}
\left(\prod_{m=1}^k\binom{I_m+J_m}{J_m}(-1)^{I_m}\right) 
p_{I+J,\0}\tilde{r}^{I}.
\end{equation*}
Together with $\alpha(i,j)\coloneq\binom{i+j}{j}(-1)^{i}$, for $|J|<n$ and 
$h\in\{1,\ldots,n\}$ we get
\begin{equation*}\begin{split}
\frac{\partial}{\partial r_h}p_J'(\tilde{r})=
&\sum_{\substack{|I|_1\leq n-|J|_1\\I_h\geq 1}}
\left(\prod_{m=1}^k\alpha(I_m,J_m)\right)p_{I+J,\0}I_h\tilde{r}^{I-\e{h}}\\
&=-(J_h+1)p_{J+\e{h}}(\tilde{r}),
\end{split}\end{equation*}
which finishes the proof.
\end{proof}
Now $\Grat(P_1,\ldots,P_k,r)$ is piecewise constant and satisfies the 
conditions of Lemma\nobreakspace \ref {lem:rqp_coeffs_poly_mink_sums}, if the sum $\sum_{i=1}^kP_i$ 
is full-dimensional. Thus, we can prove Theorem\nobreakspace \ref {thm:rqp_polycoeffs_diff_mink_sums}.

\begin{proof}[Proof of Theorem\nobreakspace \ref {thm:rqp_polycoeffs_diff_mink_sums}] 
By Theorem\nobreakspace \ref {thm:my_rational_ehrhart_mink_sum}, it is enough to show the 
statement
for 
\begin{equation*}
r\in S=[0,\denrat(P_1)]\times\ldots \times[0,\denrat(P_k)].
\end{equation*}
Since for $z\in \Z^n$ with $h(P_i,z)=\max\{z^\top x: x\in P_i\}$
we have that $z\in \sum_{i=1}^k r_iP_i$ if and only if 
$\sum_{i=1}^k r_ih(P_i,z)\geq 1$, $\Grat(P_1,\ldots,P_k,r)$ is constant on 
components of the hyperplane arrangement defined by the hyperplanes 
\begin{equation*}
\left\{\bigg\{x\in\R^n:\sum_{i=1}^k x_ih(P_i,z)=1\bigg\}:
z\in\Z^n\right\}.
\end{equation*}
For $r\in S$ it is enough to consider only  
$z\in\sum_{i=1}^k q(P_i)P_i\cap\Z^n$ since other integral points are never in 
$r_1P_1+\ldots+r_kP_k$ for $r\in S$. Thus $\Grat(P_1,\ldots,P_k,r)$ satisfies 
the conditions of Lemma\nobreakspace \ref {lem:rqp_coeffs_poly_mink_sums} for every maximal cell
of the arrangement
\begin{equation*}
\left\{\bigg\{x\in S:\sum_{i=1}^k x_ih(P_i,z)=1\bigg\}:
z\in\sum_{i=1}^k q(P_i)P_i\Z^n\right\}.
\end{equation*}
\end{proof}

\section{Proof of Theorem\nobreakspace \ref {thm:rqp_right_sides}}\label{sec:vector_dilations}

Recall that $P_A(b)\coloneq\{x\in\R^n:Ax\leq b\}$ for an integral 
$(m\times n)$-matrix $A$ and a rational vector $b\in\Q^m$ and 
\begin{equation*}
\Phi(A,b):=\#(P_A(b)\cap\Z^n),\quad b\in\Q^m.
\end{equation*}
Furthermore, we defined a $C\in\mathcal{C}_A$ to be a set of vectors $b\in\Q^k$
for which all polytopes $P_A(b)$ are combinatorially equivalent.

\begin{lemma}[{McMullen \cite[Section 2]{McMullen1973}}]\label{lem:right_side_cones}
For every $C\in\mathcal{C}_A$, the closure $\cl(C)$ is a polyhedral cone.
\end{lemma}

%
%

In general $P_A(b)+P_A(c)\subset P_A(b+c)$ for $b,c\in\Q^m$, but for $b,c\in\cl(C)$,
$C\in\mathcal{C}_A$ we have more:

\begin{lemma}[{McMullen \cite[Section 6]{McMullen1973}, Mount \cite[Theorem 2]{Mount}}]\label{lem:right_side_mink_sums}
Let $C\in\mathcal{C}_A$. For $b,c\in \cl(C)$ we have that $P_A(b)+P_A(c)=P_A(b+c)$.
\end{lemma}


Now let $C\in\mathcal{C}$, $\cl(C)=\cone\{h_1,\ldots,h_k\}$. Then, by Lemma\nobreakspace \ref {lem:right_side_mink_sums}
we have for $r,s\in\Q_{>0}^k$ that, $[r]=[s]$ if and only if 
$\sum_{i=1}^kr_ih_i=\sum_{i=1}^ks_ih_i$. Using this characterization, we can extend
the equivalence relation for all rational vectors in $\Q^k$:
For $r,s\in \Q^k$ we say, $r$ is equivalent to $s$, if $\sum_{i=1}^kr_ih_i=\sum_{i=1}^ks_ih_i$.
As before, we denote the equivalence class of $r$ by $[r]$.

Using 
Lemma\nobreakspace \ref {lem:right_side_mink_sums} and\nobreakspace Theorems\nobreakspace \ref {thm:my_rational_ehrhart_mink_sum} and\nobreakspace  \ref {thm:rqp_polycoeffs_diff_mink_sums} 
we get a structural result for $\Phi_A(b)$.

\begin{proof}[Proof of Theorem\nobreakspace \ref {thm:rqp_right_sides}]
First we fix $b\in C$ and let $\mu\in\R^k_{\geq 0}$ with $\sum_{i=1}^k\mu_ih_i=b$.
Then by 
Lemmas\nobreakspace \ref {lem:right_side_cones} and\nobreakspace  \ref {lem:right_side_mink_sums} for all 
$\lambda\in[\mu]\cap\R_{\geq 0}^k$ we have that
\begin{equation*}
\begin{split}
\#(P_A(b)\cap\Z^n)&=\#\left(P_A\left(\sum_{i=1}^m\lambda_ih_i\right)\cap\Z^n\right)\\
&=\#\left(\sum_{i=1}^m\lambda_iP_A(h_i)\cap\Z^n\right).
\end{split}
\end{equation*}
By Theorem\nobreakspace \ref {thm:my_rational_ehrhart_mink_sum} this is
\begin{equation*}
\#(P_A(b)\cap\Z^n)=\sum_{\substack{I\in\{0,\ldots,n\}^k\\|I|_1\leq n}}\Grat_I(P_A(h_1),\ldots,P_A(h_k),\lambda)\lambda^I.
\end{equation*}

To write this as a function in $b$ we need to fix a choice for $\lambda$. To do this
we show that, in fact the nonnegativity condition on $\lambda$ can be ommited:

By Corollary\nobreakspace \ref {rem:alpha_F} $\Grat_I(P_A(h_1),\ldots,P_A(h_k),\lambda)$ does not depend
on $\lambda$ but only on $[\mu]$. Thus we can write
\begin{equation*}
\sum_{\substack{I\in\{0,\ldots,n\}^k\\|I|_1\leq n}}\Grat_I(P_A(h_1),\ldots,P_A(h_k),\lambda)\lambda^I=\sum_{\substack{I\in\{0,\ldots,n\}^k\\|I|_1\leq n}}p_I([\mu])\lambda^I.
\end{equation*}
where $p_I([\mu])=\Grat_I(P_A(h_1),\ldots,P_A(h_k),\lambda)$ for all $\lambda\in[\mu]\cap\R_{\geq 0}^k$.
Thus $\#\left(P_A\left(\sum_{i=1}^m\lambda_ih_i\right)\cap\Z^n\right)$ restricted to
all $\lambda\in[\mu]\cap\R_{\geq 0}^k$ is in fact a polynomial in $\lambda$.
Since this polynomial is constant on the non-empty intersection of the affine plane 
$\sum_{i=1}^k\lambda_ih_i=b$ with the positive orthant,
we get that $\sum_{\substack{I\in\{0,\ldots,n\}^k\\|I|_1\leq n}}p_I([\mu])\lambda^I$
is constant for all $\lambda\in[\mu]$.

Hence we can write
\begin{equation*}
\#(P_A(b)\cap\Z^n)=\sum_{\substack{I\in\{0,\ldots,n\}^k\\|I|_1\leq n}}\Grat_I(P_A(h_1),\ldots,P_A(h_k),\lambda)\lambda^I
\end{equation*}
for all $\lambda\in[\mu]$ by setting $\Grat_I(P_A(h_1),\ldots,P_A(h_k),\lambda)=\Grat_I(P_A(h_1),\ldots,P_A(h_k),\mu)$.

Without loss of generality, we assume that $h_1,\ldots,h_m$ are linearly independent
and let $H=(h_1,\ldots,h_m)$. For $b\in C$ let $\lambda(b)\in\R^k$ with $(\lambda(b)_1,\ldots,\lambda(b)_m)=H^{-1}b$
and $\lambda(b)_i=0$ for $i>m$.
Then
\begin{equation*}
\begin{split}
\Phi(A,b)&=\sum_{\substack{I\in\{0,\ldots,n\}^k\\|I|_1\leq n}}\Grat_I(P_A(h_1),\ldots,P_A(h_k),\lambda(b))\lambda(b)^I\\
&=\sum_{\substack{I\in\{0,\ldots,n\}^k\\|I|_1\leq n}}
\Grat_I(P_A(h_1),\ldots,P_A(h_k),(H^{-1}b,0,\ldots,0))(H^{-1}b,0,\ldots,0)^I\\
&=\sum_{\substack{J\in\{0,\ldots,n\}^k\\|J|_1\leq n}}\Phi_J(A,b)b^J,
\end{split}
\end{equation*}
where $\Phi_J(A,b)$ is a suitably linear combination of 
$\Grat_I(P_A(h_1),\ldots,P_A(h_k),(H^{-1}b,0\ldots,0))$ for $|I|_1=|J|_1$.\\
Furthermore since $\lambda(b+\denrat(P_A(h_i))h_i)\in[\lambda(b)+\denrat(P_A(h_i))\e{i}]$ and
by Theorem\nobreakspace \ref {thm:my_rational_ehrhart_mink_sum}
\begin{equation*}
\Grat_I(P_A(h_1),\ldots,P_A(h_k),\lambda(b))=\Grat_I(P_A(h_1),\ldots,P_A(h_k),\lambda(b)+\denrat(P_A(h_i))\e{i})
\end{equation*}
we get
\begin{equation*}
\Grat_I(P_A(h_1),\ldots,P_A(h_k),\lambda(b))=\Grat_I(P_A(h_1),\ldots,P_A(h_k),\lambda(b+\denrat(P_A(h_i))h_i)).
\end{equation*}
Thus $\Phi_J(A,b)=\Phi_J(A,b+\denrat(P_A(h_i))h_i)$ for $i=1,\ldots,k$.

Since $\Grat_I(P_A(h_1),\ldots,P_A(h_m),H^{-1}b)$ is a piecewise-defined polynomial of total
degree $n-|I|_1=n-|J|_1$ in $H^{-1}b$, $\Phi_J(A,b)$ is a piecewise-defined polynomial of
total degree $n-|J|_1$ in $b$.

To show the remaining part of the theorem, we explicitely study the linear combinations of
$\Grat_I(P_A(h_1),\ldots,P_A(h_k),(H^{-1}b,0,\ldots,0))$ that results in $\Phi_J(A,b)$:
For convenience for $I,J\in\{0,\ldots,n\}^m$, we denote by 
$\mathcal{M}_1(I)$ the set of matrices $K\in\{0,\ldots,n\}^{m\times m}$
such that the row sums $\sum_{j=0}^mK_{ij}$ equal $I_i$ for all $i\in\{1,\ldots,m\}$. By $\mathcal{M}_2(I,J)$ we denote the set of those
matrices in $K\in\mathcal{M}_1(I)$, such that the column sums $\sum_{i=0}^mK_{ij}$ equal $J_j$ for all $j\in\{1,\ldots,m\}$.
For $i\in\{0,\ldots,m\}$ and $J\in\{0,\ldots,n\}^m$, we denote by $\binom{i}{J}$ the multinomial coefficient
$\frac{i!}{J_1!\cdots J_m!}$ and for $I\in\{0,\ldots,n\}^m$ and $K\in\{0,\ldots,n\}^{m\times m}$ 
by $\binom{I}{K}$ the product of multinomial coefficients $\binom{I_1}{(K_{11},\ldots,K_{1m})}\cdots\binom{I_m}{(K_{m1},\ldots,K_{mm})}$.
For $M,N\in\R^{m\times m}$  we denote by $M^{i}$ the $i$th row of $M$ and by
$M^{N}$ we denote the componentwise power $\left(M_{ij}^{N_{ij}}\right)_{ij}$.
With this notation, we have for $I\in\{0,\ldots,n\}^k$ that
\begin{equation*}
(H^{-1}b,0,\ldots,0)^I=\sum_{K\in\mathcal{M}_1(I)}\binom{I}{K}\left(H^{-1}\right)^Kb^{(\sum_{i=0}^mK_{i1},\ldots,\sum_{i=0}^mK_{im})}
\end{equation*}
if $I_{m+1}=\ldots=I_k=0$ and $(H^{-1}b,0,\ldots,0)^I=0$ otherwise.
Abbreviating $\Grat_{(I,0,\ldots,0)}(P_A(h_1),\ldots,P_A(h_k),(H^{-1}b,0,\ldots,0))$ as 
$\Grat_{I}(A,b)$ we get for $J\in\{0,\ldots,n\}^k$
\begin{equation*}
\Phi_J(A,b)=\sum_{\substack{I\in\{0,\ldots,n\}^m\\|I|_1=|J|_1}}\Grat_I(A,b)\sum_{K\in\mathcal{M}_2(I,J)}\binom{I}{K}\left(H^{-1}\right)^K.
\end{equation*}
With Theorem\nobreakspace \ref {thm:rqp_polycoeffs_diff_mink_sums} and the multi-dimensional chain rule for
derivatives, we have
\begin{equation*}\begin{split}
\frac{\partial}{\partial b_l}\Phi_J(A,b)&=\sum_{\substack{I\in\{0,\ldots,n\}^m\\|I|_1=|J|_1}}
\left(\sum_{g=1}^m-(I_g+1)\Grat_{I+\e{g}}(A,b)\left(H^{-1}\right)_{gl}\right)\sum_{K\in\mathcal{M}_2(I,J)}\binom{I}{K}\left(H^{-1}\right)^K\\
&=-\sum_{g=1}^m\sum_{\substack{I\in\{0,\ldots,n\}^m\\|I|_1=|J|_1}}
\Grat_{I+\e{g}}\sum_{K\in\mathcal{M}_2(I,J)}(I_g+1)\binom{I}{K}\left(H^{-1}\right)^{K+\E{g}{l}}.
\end{split}\end{equation*}
By shifting the index $I_g$ by one we finally find
\begin{equation*}\begin{split}
\frac{\partial}{\partial b_l}\Phi_J(A,b)&=-\sum_{\substack{I\in\{0,\ldots,n\}^m\\|I|_1=|J|_1+1}}
\Grat_{I}(A,b)\sum_{K\in\mathcal{M}_2(I,J+e(l))}\left(\sum_{g=1}^m(I_g)\binom{I-e(g)}{K-\E{g}{l}}\right)\left(H^{-1}\right)^{K}\\
&=-\sum_{\substack{I\in\{0,\ldots,n\}^m\\|I|_1=|J|_1+1}}
\Grat_{I}(A,b)\sum_{K\in\mathcal{M}_2(I,J+e(l))}\left((J_l+1)\binom{I}{K}\right)\left(H^{-1}\right)^{K}\\
&=-(J_l+1)\Phi_{J+e(l)}(A,b).
\end{split}\end{equation*}
\end{proof}

Since the leading term of $\Grat(P_A(h_1),\ldots,P_A(h_k),\lambda)$ is the volume
of $\sum_{i=1}^k\lambda_iP_A(h_i)$ we also get, that the leading term of
$\Phi_A(b)$ is the volume of $P_A(b)$. This implies, that $\vol(P_A(b))$ is a
homogeneous polynomial of degree $n$ in $b$. We refer to \cite{Xu2011} for a
closed formula of this polynomial.

\begin{example}
As an example, we consider the following polytope in dimension $n=2$ with $m=4$
inequalities:
\begin{equation*}
\begin{split}
P(a,bc,d)=\{\;(x,y)^\top\in\R^2:\quad\begin{array}[t]{rcl}2x+y&\leq&a,\\ 
-2x+y&\leq&b,\\
y&\leq&c,\\
-y&\leq&d\quad\}.
\end{array}
\end{split}
\end{equation*}
\begin{figure}[hbt]\centering
\includegraphics{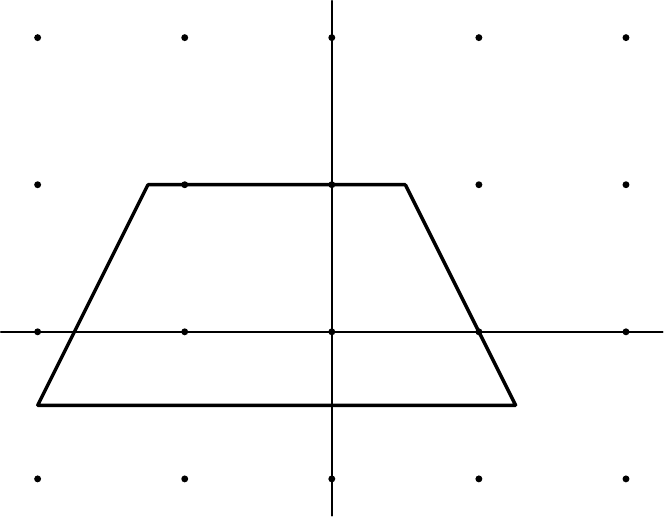}
\caption{$P_A(2,\frac72,1,\frac12)$}\label{fig:rqp_right_side} 
\end{figure}

The intersection point of both non-horizontal inequalities is $v=\left(\frac{a-b}4,\frac{a+b}2\right)$.
This polytope is nonempty, whenever $-d\leq c$ and $v_2\geq -d$, 
that is $a+b+2d\geq 0$. If $v_2\leq c$, that is $a+b-2c\leq 0$, then $P_A(a,b,c,d)$
is a triangle. Otherwise, that is, if $a+b-2c> 0$, $P$ is a proper quadrangle.
Hence, there are the following possible cones in $\mathcal{C}_A$:
\[C_{\text{point}}=\{(a,b,c,d)\in\R^4: c+d\geq 0,a+b+2d=0\}\] for a single point,
\[C_{\text{line}}=\{(a,b,c,d)\in\R^4:c+d= 0,a+b+2d>0\}\]  for a line-segment, 
\[C_{3\text{-gon}}=\{(a,b,c,d)\in\R^4: c+d>0,a+b+2d>0,a+b-2c\leq 0\}\] for the triangle and 
\[C_{4\text{-gon}}=\{(a,b,c,d)\in\R^4: c+d>0,a+b+2d>0,a+b-2c>0\}\] for the quadrangle.

Since $C_{\text{point}}\subset\cl(C_{3\text{-gon}})$ and $C_{\text{line}}\subset\cl(C_{4\text{-gon}})$
it is sufficient to investigate $\cl(C_{3\text{-gon}})$ and $\cl(C_{4\text{-gon}})$.

Here we get
\begin{equation*}
\begin{split}
\Phi&\left(A,(a,b,c,d)^\top\right)=\frac12\left(d^2-c^2+bc+ac+bd+ad\right)+\frac{a}{2}\left(1-\fracpart{d}-\fracpart{c}\right)\\
&+\frac{b}2\left(1-\fracpart{d}-\fracpart{c}\right)
+\frac{c}2\left(-\fracpart{b}+2\fracpart{c}-\fracpart{a}\right)
+\frac{d}2\left(2-2\fracpart{d}-\fracpart{b}-\fracpart{a}\right)\\
&+\bigg(\fracpart{\frac{c-a}{2}}^2+\fracpart{\frac{c-b}{2}}^2-\fracpart{\frac{a+d}{2}}^2-\fracpart{\frac{b+d}{2}}^2\\
&+\fracpart{\frac{a+d}{2}}\fracpart{a}+\fracpart{\frac{b+d}{2}}\fracpart{b}+\fracpart{\frac{c-a}{2}}\fracpart{a}
+\fracpart{\frac{c-b}{2}}\fracpart{b}\\
&-\fracpart{\frac{c-a}{2}}\fracpart{c}-\fracpart{\frac{c-b}{2}}\fracpart{c}
+\fracpart{\frac{b+d}{2}}\fracpart{d}+\fracpart{\frac{a+d}{2}}\fracpart{d}\\
&-\fracpart{\frac{c-b}{2}}-\fracpart{\frac{c-a}{2}}-\fracpart{a}-\fracpart{b}
+\fracpart{c}-\fracpart{d}+1\bigg)
\end{split}
\end{equation*}
for all $(a,b,c,d)\in\cl(C_{4\text{-gon}})$ and
\begin{equation*}
\begin{split}
\Phi&\left(A,(a,b,c,d)^\top\right)=\frac18\left(a^2+b^2+4d^2+2ab+4ad+4bd\right)\\
&+\frac{a}4\left(2-\fracpart{a}-\fracpart{b}-2\fracpart{d}\right)+\frac{b}4\left(2-\fracpart{a}-\fracpart{b}-2\fracpart{d}\right)\\
&+\frac{d}2\left(2-\fracpart{a}-\fracpart{b}-2\fracpart{d}\right)\\
&+\bigg(2\fracpart{\frac{a-b}4}^2-\fracpart{\frac{b+d}2}^2-\fracpart{\frac{a+d}2}^2
+\fracpart{a}\fracpart{\frac{a+d}2}+\fracpart{b}\fracpart{\frac{b+d}2}\\
&-\fracpart{a}\fracpart{\frac{a-b}4}-\fracpart{b}\fracpart{\frac{b-a}4}
+\fracpart{d}\fracpart{\frac{b+d}2}+\fracpart{d}\fracpart{\frac{a+d}2}\\&
-\fracpart{d}-2\fracpart{\frac{a-b}4}+1\bigg)
\end{split}
\end{equation*}
for all $(a,b,c,d)\in\cl(C_{3\text{-gon}})$.

\end{example}

\providecommand{\bysame}{\leavevmode\hbox to3em{\hrulefill}\thinspace}
\providecommand{\MR}{\relax\ifhmode\unskip\space\fi MR }
\providecommand{\MRhref}[2]{%
  \href{http://www.ams.org/mathscinet-getitem?mr=#1}{#2}
}
\providecommand{\href}[2]{#2}

\end{document}